\documentclass[ 11pt]{article} 
\usepackage[utf8]{inputenc}
\usepackage[babel]{microtype}
\usepackage[english]{babel}

\usepackage{amsmath,amssymb,amsthm, mathrsfs, mathtools,bbm,bm}
\usepackage[usenames,dvipsnames]{xcolor}
\usepackage{multirow}
\usepackage{enumerate}
\usepackage{enumitem}
\setlist[enumerate,1]{label={(\roman*)}}
\setlist[enumerate,2]{label={(\alph*)}}
\setlist[enumerate,3]{label={(\arabic*)}}
\setlist[itemize]{nolistsep,noitemsep, topsep=0pt}
\usepackage[bmargin=20mm,tmargin=20mm,lmargin=20mm,rmargin=20mm]{geometry}
\usepackage[hypertexnames=false]{hyperref}
\hypersetup{
    colorlinks,
    linkcolor={red!60!black},
    citecolor={green!60!black},
    urlcolor={blue!60!black}
}

\newcommand{\cA}{\ensuremath{\mathcal A}}
\newcommand{\cB}{\ensuremath{\mathcal B}}
\newcommand{\cC}{\ensuremath{\mathcal C}}
\newcommand{\cD}{\ensuremath{\mathcal D}}

\newcommand{\cF}{\ensuremath{\mathcal F}}

\newcommand{\cH}{\ensuremath{\mathcal H}}

\newcommand{\cK}{\ensuremath{\mathcal K}}

\usepackage[capitalize]{cleveref}
\crefname{equation}{}{}

\theoremstyle{plain}
\newtheorem{theorem}{Theorem}[section]

\newtheorem{proposition}[theorem]{Proposition}
\newtheorem{lemma}[theorem]{Lemma}
\newtheorem{corollary}[theorem]{Corollary}
\newtheorem{problem}[theorem]{Problem}
\theoremstyle{definition}

\newtheorem{remark}[theorem]{Remark}

\DeclarePairedDelimiter{\parens}{(}{)}
\DeclarePairedDelimiter{\set}{\{}{\}}

\DeclarePairedDelimiter{\card}{\lvert}{\rvert}

\DeclarePairedDelimiter{\floor}{\lfloor}{\rfloor}
\DeclarePairedDelimiter{\ceil}{\lceil}{\rceil}

\DeclareMathOperator{\wt}{wt}
\DeclareMathOperator{\supp}{supp}

\newcommand{\F}{\ensuremath{\mathbb{F}}}

\newcommand{\Erdos}{Erd\H{o}s}
\newcommand{\Gyarfas}{Gy\'arf\'as}

\newcommand{\Szekeres}{Szekeres}

\usepackage{graphicx}
\graphicspath{{images/}} 
\input{tikz}

\allowdisplaybreaks

\begin{document}

\author{
Simona Boyadzhiyska\,\footnotemark[1] \qquad 
Shagnik Das\,\footnotemark[2] \qquad
Thomas Lesgourgues\,\footnotemark[3] \\
Kalina Petrova\,\footnotemark[4]
}

\footnotetext[1]{School of Mathematics, University of Birmingham, Edgbaston, Birmingham, B15 2TT, UK.  Email: {\tt s.s.boyadzhiyska@bham.ac.uk}}
\footnotetext[2]{Department of Mathematics, National Taiwan University, Taipei 10617, Taiwan. Email: {\tt shagnik@ntu.edu.tw}}
\footnotetext[3]{Department of Combinatorics and Optimization, University of Waterloo, Canada. Email: {\tt tlesgourgues@uwaterloo.ca}.}
\footnotetext[4]{Institute of Science and Technology Austria (ISTA), Klosterneurburg 3400, Austria. Email: {\tt kalina.petrova@ist.ac.at}.}

\sloppy

\title{Odd-Ramsey numbers of complete bipartite graphs}

\maketitle

\begin{abstract}
In his study of graph codes, Alon introduced the concept of the \emph{odd-Ramsey} number of a family of graphs $\cH$ in $K_n$, defined as the minimum number of colours needed to colour the edges of $K_n$ so that every copy of a graph $H\in \cH$ intersects some colour class in an odd number of edges. In this paper, we focus on complete bipartite graphs. First, we completely resolve the problem when $\cH$ is the family of all spanning complete bipartite graphs on~$n$ vertices. We then focus on its subfamilies, that is, $\{K_{t,n-t}\colon t\in T\}$ for a fixed set of integers $T\subseteq [\floor{n/2}]$. We prove that the odd-Ramsey problem is equivalent to determining the maximum dimension of a linear binary code avoiding codewords of given weights, and leverage known results from coding theory to deduce asymptotically tight bounds in our setting. We conclude with bounds for the odd-Ramsey numbers of fixed (that is, non-spanning) complete bipartite subgraphs.
\end{abstract}

\section{Introduction}

The classic problem in Ramsey Theory asks for the smallest integer $n$ such that every $r$-colouring of the edges of the complete graph~$K_n$ contains a monochromatic copy of $K_k$. Despite nearly a century of research, this problem remains poorly understood, even for small values of~$r$ or~$k$. In the $2$-colour case, the early work of \Erdos{}~\cite{erdos1947} and \Erdos{} and \Szekeres{}~\cite{erdos_combinatorial_1935} in the 1940s and 1930s respectively established exponential (in $k$) lower and upper bounds for $n$. Decades of subsequent efforts only saw lower order improvements to these bounds, until the very recent breakthrough of Campos, Griffiths, Morris, and Sahasrabudhe~\cite{campos2023exponential} (further optimised  by Gupta, Ndiaye, Norin, and Wei~\cite{gupta2024optimizing}) brought the first exponential improvement on the upper bound of~\Erdos{} and~\Szekeres{}. The gaps in the bounds are even wider at the other end of the spectrum, where we seek monochromatic triangles as the number of colours $r$ grows. The best known lower bound of $(321 - o(1))^{r/5}$ comes from blowing up a $5$-colouring of Exoo~\cite{exoo1994lower}, while the upper bound of $e r!$ follows from an argument of Schur~\cite{schur1916uber}. Determining the correct order of growth of this multicolour Ramsey number is a famous old problem of \Erdos{} with a \$250 prize attached~\cite{chung1997open}.

Even with the large gaps between the upper and lower bounds, Ramsey Theory has found many applications and, as a result, several natural variants of the problem have been considered over the years. One of these is the \emph{generalised Ramsey number} $f(G,H,q)$, defined to be the minimum number of colours needed in an edge-colouring of $G$ with the property that every copy of $H$ in $G$ receives at least $q$ different colours. Note that the classic Ramsey problem is concerned with avoiding monochromatic copies of a graph $H$, and thus corresponds to the case $q = 2$. This generalisation of requiring more colours in each copy of $H$ has a long and rich history too, including in particular the seminal article of Erd\H{o}s and Gy\'arf\'as~\cite{erdHos1997variant}, and the 
determination of the polynomial threshold for $f(K_n,K_p,q)$ by Conlon, Fox, Lee, and Sudakov~\cite{conlon2015erdHos}, with previous progress appearing in~\cite{eichhorn2000note,fox2009ramsey,mubayi1998edge}. We refer the reader to the recent breakthrough of Bennett, Cushman, Dudek, and Pra{\l}at~\cite{bennett2024erdHos} on $f(K_n,K_4,5)$, and of Bennett, Cushman, and Dudek~\cite{bennett2024generalized} on $f(K_n,K_5,8)$ for further references. In this article, we briefly use results on the bipartite case $f(K_n,K_{s,t},q)$, initiated by Chv\'{a}tal~\cite{chvatal1969finite} in relation to the Zarankiewicz problem, and studied further by Axenovich, F{\"u}redi, and Mubayi~\cite{axenovich2000generalizedbip}. 

Aside from simply counting the number of colours, one can further generalise this problem by imposing other conditions on the colour patterns appearing on copies of~$H$ in edge-colourings of a host graph~$G$. In his work on graph codes (which are families of graphs whose pairwise symmetric differences do not lie in a given collection of graphs $\cH$), Alon~\cite{alon2024graph} showed that bounds on the size of graph codes can be obtained by studying the parity of colour patterns on graphs in~$\cH$ in edge-colourings of $K_n$ (see Versteegen~\cite{versteegen2023upper} for a concise explanation of the connection between the problems). To that end, Alon defined the \emph{odd-Ramsey number} $r_{odd}(n, \cH)$ to be the minimum number of colours $r$ needed in an edge-colouring $G_1, G_2, \hdots, G_r$ of~$K_n$ with the property that every copy of a graph from $\cH$ intersects some~$G_i$ in an odd number of edges. For simplicity, we say that every copy of a graph from $\cH$ has an \emph{odd colour class}, and when $\cH = \set{H}$ consists of a single graph, we write $r_{odd}(n,H)$ instead.

Alon implicitly addressed the odd-Ramsey numbers of stars, matchings, and certain families of cliques in the proofs of Theorems~1.2 and~1.6 and Proposition~1.3 of~\cite{alon2024graph}. Many papers have subsequently considered the case of complete graphs, seeking to determine $r_{odd}(n,K_t)$, as that would give lower bounds on the size of the maximum graph codes of cliques, which Alon inquired about. Cameron and Heath~\cite{cameron2023new} proved that $r_{odd}(n,K_4) = n^{o(1)}$, after which Bennett, Heath, and Zerbib~\cite{bennett2023edgecoloring} and, independently, Ge, Xu, and Zhang~\cite{ge2023new} showed that $r_{odd}(n,K_5) = n^{o(1)}$. The latter set of authors also proved that $r_{odd}(n,K_5) = \Omega(\log n)$, and conjectured that $r_{odd}(n,K_t) = n^{o(1)}$ for all $t$; the former set of authors proved some general upper bounds in this case. Finally, Yip~\cite{yip2024k8} showed that $r_{odd}(n,K_8) = n^{o(1)}$.

Versteegen~\cite{versteegen2023upper}  considered the odd-Ramsey numbers of general graphs $H$. Note that the problem is trivial for graphs $H$ with an odd number of edges, as they will have an odd colour class in any edge-colouring. Thus, we may restrict our attention to graphs $H$ with an even number of edges. In this case, Versteegen proved the general lower bound of $r_{odd}(n,H) = \Omega( \log n)$ for any fixed graph $H$, and showed that this could be strengthened to the polynomial lower bound $r_{odd}(n,H) \ge n^{1/(v(H) - 2)}$ if $H$ can be decomposed into independent sets in a particular way. Versteegen showed that almost all graphs with an even number of edges satisfy this property, and Janzer and Yip~\cite{janzer2024} recently proved sharp bounds on the probability of the random graph~$G(n,p)$ having such a decomposition.

\subsection{Results}

In this paper, we study the odd-Ramsey numbers of complete bipartite graphs. We obtain results for the family of all spanning complete bipartite graphs on $n$ vertices, its subfamilies, and for individual fixed graphs $K_{s,t}$. In the process, we establish a somewhat surprising connection between odd-Ramsey numbers for spanning bipartite graphs and a parameter studied in coding theory.
\medskip

Our first result, \cref{prop:odd_ramsey_unbalanced_family}, determines precisely the odd-Ramsey number for the family of all spanning complete bipartite graphs on $n$ vertices. 
\begin{theorem}\label{prop:odd_ramsey_unbalanced_family}
Let $\cF = \{ K_{t,n-t}\colon t \in [n/2]\}$ be the family of all complete bipartite graphs on $n$ vertices. Then
\[r_{odd}(n,\cF) = \begin{cases}
    n-1 &\qquad\text{ if $n$ is even},\\
    n &\qquad\text{ if $n$ is odd}.
    \end{cases}\]
\end{theorem}\medskip

We next focus on odd-Ramsey numbers of subfamilies of $\cF$. Given an integer $n \ge 1$ and a set $T \subseteq [n/2]$, we denote by $\cF_T$ the family of spanning complete bipartite graphs $\set{K_{t,n-t} : t \in T}$.

\begin{remark}\label{rem:OnlyEvenGraphs}
We emphasise once more that graphs with an odd number of edges are irrelevant in the determination of odd-Ramsey numbers. Indeed, let $\cH$ be a family of graphs, and suppose $H \in \cH$ has an odd number of edges. This trivially implies $H$ has an odd colour class in every colouring, and hence the copies of~$H$ do not pose any restrictions on the edge-colouring of $K_n$, and so we have $r_{odd}(n,\cH) = r_{odd}(n,\cH \setminus \set{H})$.

In particular, given $T \subseteq [n/2]$, if we set $S = \set{t \in T: t(n-t) \in 2\mathbb{Z}}$, then $r_{odd}(n, \cF_T) = r_{odd}(n, \cF_S)$. Thus, we shall henceforth assume $T$ only contains integers $t$ for which $t(n-t)$ is even. That is, if $n$ is even, then every $t \in T$ must also be even.
\end{remark}

We prove our results by establishing an equivalence between our odd-Ramsey problem and a well-known problem from coding theory. We briefly recall some basic terminology: a $k$-dimensional {\em linear binary code} of length $n$ is a $k$-dimensional subspace $\cC\subseteq \F_2^n$, and we refer to the elements of $\cC$ as {\em codewords}. The {\em(Hamming) weight} of a codeword $\mathbf{v}$, denoted $\wt(\mathbf{v})$, is the size of its support. 

For a given set $S\subseteq [n]$, we define $\ell(n,S)$ to be the largest dimension of any linear binary code~$\cC$ of length~$n$ with no codewords whose weight is in $S$; that is,
\[\ell(n,S)=\max\set*{\dim(\cC)\,\colon\,  \cC\subseteq \mathbb{F}^n_2\text{ is linear, } \wt(\mathbf{v})\not\in S \text{ for all } \mathbf{v}\in \cC}.\]
This parameter has been studied extensively for certain choices of $S$, see for instance~\cite{enomoto1987codes,mazorow1991extremal}; we collect some relevant results in~\cref{sec:PrelimCode}. Surprisingly, this function turns out to be closely related to the odd-Ramsey number of~$\cF_T$ in~$K_n$.

\begin{theorem}\label{thm:BipartiteGeneralEquivalence}
    For any integer $n\geq 1$ and set $T\subseteq [n/2]$ such that $t(n-t)$ is even for all $t\in T$, let~$W_T$ be the set $T\cup\{n-t\colon t\in T\}\cup\{n\}$. Then we have
        \[r_{odd}(n,\cF_T) = n-\ell(n,W_T).\]    
\end{theorem}
In particular, this theorem implies that, for any integers $n>t\geq1$ with $t(n-t)$ even, we have $r_{odd}(n,K_{t,n-t}) = n-\ell(n,\{t,n-t,n\})$. The main intuition behind this relation is that, given an $r$-colouring of $K_n$, one can build a $(r-1)$-dimensional binary linear code based on the vertices of odd degree in each colour. Asking for every copy of $K_{t,n-t}$ to contain an odd colour class turns out to then be equivalent to forbidding codewords of weights $t$ or $n-t$ in the dual of this code. With some care, we can reverse the construction and build colourings from binary codes.

Combining \cref{thm:BipartiteGeneralEquivalence} with appropriate results from coding theory allows us to derive concrete bounds on $r_{odd}(n,\cF_T)$. We highlight the most natural cases in the corollary below. In particular, we resolve the problem up to an additive constant when $n$ is large and $\max T$ is bounded by a constant.

\begin{corollary}\label{cor:SpecialCases}
    Let $T \subseteq [n/2]$ be such that $t(n-t)$ is even for every $t \in T$. 
    \begin{enumerate}[label=(\alph*)]
        \item If $T$ only contains odd integers, then $r_{odd}(n,\cF_T) = t+2$, where $t = \max T$. \label{SpecialCases:odd}
        \item If $T$ contains an even integer, let $2d$ be the largest even integer in $T$, and let $t = \max T$.\label{SpecialCases:even}
        \begin{enumerate}[label=(b\arabic*)]
            \item There is a constant $C_d$ depending only on $d$ such that \label{SpecialCases:asymp}
                \[ d \log n - C_d \le r_{odd}(n, \cF_T) \le d \log n + d + t + 2. \]
            \item Suppose $T = \{2d\}$. Then, if $n = 4d$, we have $r_{odd}(n, \cF_T) = 2d+1$. Otherwise, if $n > 4d$, we have $2d+2 \le r_{odd}(n, \cF_T) \le n-2d+1$.\label{SpecialCases:evensmall}
        \end{enumerate}
    \end{enumerate}
\end{corollary}

In particular, when $\mathcal{F}_T$ consists of a single graph $K_{t,n-t}$, we know $r_{odd}(n,K_{t,n-t})$ precisely when $t$ is odd or when $t = n/2$. More generally, we determine $r_{odd}(n,\cF_T)$ asymptotically when the set $T$ is fixed and~$n$ tends to infinity. Observe that, combining parts~\ref{SpecialCases:odd} and~\ref{SpecialCases:asymp} and the monotonicity of odd-Ramsey numbers, we obtain asymptotically tight bounds on $r_{odd}(n,\cF_T)$ for any constant~$d$ and any $t\notin\Theta(\log n)$; if $t=\Theta(\log n)$, our bounds show that the corresponding odd-Ramsey number is also $\Theta(\log n)$, but the constant factors in the two bounds differ  (see~\cref{problem:ExactSpanning} and the preceding discussion).\medskip

At the other end of the spectrum, one could ask for the odd-Ramsey numbers of fixed (that is, non-spanning) complete bipartite subgraphs instead. That is, what is $r_{odd}(n,K_{s,t})$? Naturally, we are only interested in this problem when $st$ is even; if both $s$ and $t$ are odd, then a single colour trivially suffices. Let us thus assume $s \le t$ and $st$ is even. As observed by Versteegen~\cite{versteegen2023upper}, we can obtain a lower bound via the K\H{o}v\'ari--S\'os--Tur\'an Theorem~\cite{kovari1954problem}. By their theorem, any subgraph of $K_n$ with at least $C_{s,t} n^{2-\frac{1}{s}}$ edges must contain a copy of $K_{s,t}$. Thus, if we have a colouring of $K_n$ with at most $cn^{\frac{1}{s}}$ colours, the densest colour class must contain a copy of the complete bipartite graph, and a monochromatic $K_{s,t}$ is an even~$K_{s,t}$. We therefore have 
\[ r_{odd}(n,K_{s,t}) \ge c n^{\frac{1}{s}}. \] 

For an upper bound, we can appeal to results on the generalised Ramsey number. Recall that $f(G,H,q)$ denotes the minimum number of colours needed in a colouring of $G$ with the property that every copy of~$H$ receives at least $q$ colours. 
Note that, if $K_{s,t}$ is coloured with at least $\frac12 st + 1$ colours, then there must be a colour that appears exactly once. Hence, using a general upper bound of Axenovich, F{\"u}redi, and Mubayi~\cite[Theorem~3.2]{axenovich2000generalizedbip}, we have
\[ r_{odd}(n,K_{s,t}) \le f(K_n, K_{s,t}, \tfrac12 st + 1) = O \parens*{n^{\frac{2s+2t-4}{st}}}. \]

These existing bounds leave us with a constant factor gap in the exponents. Our next result, which we prove by modifying the K\H{o}v\'ari--S\'os--Tur\'an proof for our purposes, goes some way towards narrowing this gap.
\begin{theorem}\label{thm:FixedBipGraphs}
    Let $s \le t$ be integers with $st$ even. Then
    \[ r_{odd}(n,K_{s,t})\geq (1 + o(1)) \parens*{\frac{n}{t}}^{1/\ceil{s/2}}\]
\end{theorem}
When $s = 2$, the exponent of our lower bound matches that of the upper bound, showing that $r_{odd}(n, K_{2,t})$ is linear in $n$. This was previously known only for $s=t=2$, that is, for $4$-cycles, following the work of~\Erdos{} and~\Gyarfas{}~\cite{erdHos1997variant} on $f(n,4,5)$ and using a simple pigeonhole argument for the lower bound (see~e.g.~\cite[Section~4]{bennett2023edgecoloring}).

\paragraph{Organisation of the paper.} The rest of the paper is organised as follows. In \cref{sec:FamilyAllCompBip}, we prove \cref{prop:odd_ramsey_unbalanced_family}, which determines the odd-Ramsey number of the family of all spanning complete bipartite graphs. \cref{sec:SubFamiliesCompBip} focuses on subfamilies of that family, showing \cref{thm:BipartiteGeneralEquivalence} and \cref{cor:SpecialCases}. We next move to fixed-size bipartite graphs with the proof of \cref{thm:FixedBipGraphs} in \cref{sec:fixed_complete_bipartite_graphs}. Finally, we discuss some open problems and a further variant of odd-Ramsey numbers in \cref{sec:concluding_remarks}.

\paragraph{Notation.} 
Our graph notation is mostly standard. Unless otherwise specified, the vertex set of a graph on $n$ vertices is $[n]=\set*{1,\ldots,n}$. For simplicity, we denote by $[n/2]$ the set $[\floor{n/2}]=\{1,2,\ldots,\floor{n/2}\}$, irrespective of the parity of $n$. We often abuse notation and identify a graph with its edge set. For a graph $G$ and two disjoint sets of vertices $A,B\subseteq V(G)$, we write $e_G(A,B)$ for the number of edges in the bipartite subgraph of $G$ with vertex classes $A$ and $B$. For a subset $A\subseteq V(G)$, we write $e_G(A)$ for the number of edges of $G$ with both endpoints in $A$. We sometimes write $A\sqcup B$ instead of $A\cup B$ when we want to emphasise that the sets $A$ and $B$ are disjoint. 

For a vector $\mathbf{a} = (a_1,\dots, a_n)\in \F_2^n$, the {\em support} of $\mathbf{a}$, which we denote by $\supp(\mathbf{a})$, is the set of all indices corresponding to nonzero entries in $\mathbf{a}$, i.e., $\supp(\mathbf{a})=\set{i\colon a_i\neq 0}$. For a set $A\subseteq [n]$, the \emph{characteristic vector} $\mathbf{a}\in \F_2^n$ of the set $A$ is the vector whose support is $A$. Additionally, we denote the all-one vector of length $n$ by~$\mathbf{1}_n$, and the all-zero vector of length $n$ by~$\mathbf{0}_n$. Given vectors $\mathbf{x}=(x_1,\dots,x_n), \mathbf{y}=(y_1,\dots, y_n)$ in $\F_2^n$, we write $\mathbf{x}\cdot\mathbf{y}$ for the usual dot product in $\F_2^n$, that is, $\mathbf{x}\cdot\mathbf{y} = x_1y_1+\hdots+ x_ny_n$. All logarithms are to base $2$.

\section{All spanning complete bipartite graphs}\label{sec:FamilyAllCompBip}

We turn our attention to the family  {$\cF = \{ K_{t,n-t}\colon t \in [n/2]\}$} of all spanning complete bipartite graphs, proving~\cref{prop:odd_ramsey_unbalanced_family}.

\begin{proof}[Proof of \cref{prop:odd_ramsey_unbalanced_family}]
We  begin with  the upper bounds. Suppose first that $n$ is odd; our goal is to show that $r_{odd}(n,\cF)\leq n$. That is, we must prove the existence of an $n$-colouring of $K_n$ such that every copy of~$K_{t,n-t}$ with $t \in [n/2]$ contains an odd colour class. Take a Hamiltonian path~$P$ and give each of its edges a distinct colour, and then assign the remaining colour to all edges in $E(K_n) \setminus E(P)$. For any $t \in [n/2]$ and any copy $K$ of $K_{t,n-t}$ in $K_n$, at least one of the edges in $P$ must go between the parts of~$K$, and is then the unique edge of its colour in~$K$. This edge thus forms an odd colour class.

Now suppose that $n$ is even; the desired bound in this case is $r_{odd}(n,\cF) \le n-1$. We modify the colouring above to save a colour by replacing the path $P$ with a path $P'$ of length $n-2$ instead. Now consider any spanning complete bipartite graph $K$. If $K$ contains any edge of~$P'$, then we are done as before, and hence we may assume that all $n-1$ vertices of $P'$ are in the same part of~$K$. Thus,~$K$ is a copy of $K_{1,n-1}$, with all of its edges coming from $E(K_n) \setminus E(P')$. Therefore,~$K$ is monochromatic, and as $n$ is even, $K$ has an odd number of edges, so we again have an odd colour class.\medskip

For the lower bound, we modify the proof of~\cite[Theorem~1.6]{alon2024graph}. First suppose that $n$ is even, and let $G_1, \dots, G_{n-2}$ be any $(n-2)$-colouring of $K_n$. We shall find a nontrivial partition $A\sqcup B$ of $[n]$ such that every $G_s$ contains an even number of edges between $A$ and $B$. Recall the well-known Chevalley--Warning theorem (see e.g.~\cite{borevich1986number,schmidt2006equations}): over a finite field, if a system of polynomial equations in $n$ variables, whose sum of degrees is at most~$n-1$, admits a solution, then it also admits a second one. As in~\cite{alon2024graph}, we associate to each vertex $i$ a variable $x_i$ over $\mathbb{F}_2$, interpreting $x_i=0$ or $x_i=1$ to mean that the vertex $i$ belongs to $A$ or to $B$, respectively. Consider the following system of $n-1$ linear equations over $\mathbb{F}_2$. For each $1 \leq s \leq n-2$, we have the equation
$$ \sum_{ij \in E(G_s)} (x_i + x_j) = 0.$$
Additionally, we also have the equation
$$ \sum_{i=1}^{n-1} x_i = 0.$$
There are $n$ variables in this system of equations, and the sum of the degrees of these polynomials is~$n-1$, so the Chevalley--Warning theorem applies. Note that $x_i=0$ for all $i\in [n]$ is a solution of the system, implying that another solution must exist. Because of the equation $\sum_{i=1}^{n-1} x_i = 0$ and $n$ being even, the assignment $x_i=1$ for all $i\in [n]$ is not a solution. Thus, there must be a solution with $x_i =0$ for at least one $i$ and $x_j = 1$ for at least one $j$. Consider a spanning complete bipartite whose parts are the support of~$x$ and its complement. For each $G_s$ and each edge $ij\in E(G_s)$, we have $x_i + x_j = 1$ if and only if $x_i \neq x_j$, which happens if and only if $ij$ belongs to the copy of $K_{t,n-t}$ we are considering. Thus, the equations $ \sum_{ij \in E(G_s)} (x_i + x_j) = 0$ imply that each $G_s$ intersects the complete bipartite graph between $A$ and $B$ in an even number of edges.

If $n$ is odd, we modify this lower bound by considering $G_1, \dots, G_{n-1}$ and still showing that there is a copy of $K_{t,n-t}$ that contains an even number of edges from each $G_s$. We have the variables~$x_i$ over $\mathbb{F}_2$ for each vertex $i$ as before, and consider a very similar system of equations. For each $G_s$ with $1 \leq s \leq n-2$, we have the equation
$$ \sum_{ij \in E(G_s)} (x_i + x_j) = 0.$$
This is as before, except that we did not introduce an equation associated with our last graph~$G_{n-1}$. The last equation in our system is now 
$$ \sum_{i=1}^n x_i=0,$$
again ensuring that $x_i = 1$ for all $i$ is not a solution. The sum of the degrees of the polynomials is still~$n-1$, smaller than the number of variables. Since $x_i=0$ for all $i$ is still a solution, there must be another one, which again corresponds to a spanning complete bipartite graph. As before, for each $s \in [n-2]$, the graph~$G_s$ intersects $K$ in an even number of edges. Since $n$ is odd, the number of edges $t(n-t)$ of~$K_{t,n-t}$ is even. The colouring $G_1, \dots, G_{n-1}$ partitions the edges of~$K$, so since each $G_s$ with $s \leq n-2$ intersects~$K_{t,n-t}$ in an even number of edges, $G_{n-1}$ does as well.
\end{proof}

\section{Some spanning complete bipartite graphs}\label{sec:SubFamiliesCompBip}

We now focus on subsets $T \subseteq [n/2]$ and ask for a colouring of $K_n$ such that, for each $t \in T$, every copy of~$K_{t,n-t}$ contains an odd colour class. We first prove~\cref{thm:BipartiteGeneralEquivalence}, showing that $r_{odd}(n, \cF_T)$ relates to a problem from coding theory. We start by introducing some key concepts required for the proof of this statement.

Recall that a {\em linear binary code} of length $n$ and dimension~$k$ is a linear subspace $\cC$ of dimension $k$ in~$\mathbb{F}^n_2$. The \emph{weight} of a codeword $\mathbf{v}$, denoted by $\wt(\mathbf{v})$, is the number of ones it contains. A binary code is called \emph{even} if it contains only even-weight vectors.  The \emph{dual code} of $\cC$, denoted $\cC^{\perp}$, is defined to be the orthogonal complement of $\cC$ in $\F_2^n$, that is, 
\( \cC^{\perp} = \{\mathbf{x} \colon \mathbf{x} \cdot \mathbf{c} = 0\text{ for all } \mathbf{c} \in \cC \}.\)
It follows that the dual of a linear code of length $n$ and dimension $k$ is itself a linear code of length $n$ and dimension $n-k$. Moreover, a linear binary code is even if and only if its dual contains the vector $\mathbf{1}_n=(1,\ldots,1)$.

\subsection{General equivalence result}

We are now ready to prove~\cref{thm:BipartiteGeneralEquivalence} in two steps, first proving a lower bound on $r_{odd}(n,\cF_T)$. Recall that, given a set $T\subseteq[n/2]$, we write $W_T$ for the set $T\cup\{n-t\colon t\in T\}\cup\{n\}$.

\begin{proposition}\label{prop:GeneralEquivalence_Lower}
    For any integer $n\geq 1$ and set $T\subseteq [n/2]$ such that $t(n-t)$ is even for all $t\in T$, we have
    \[r_{odd}(n,\cF_T) \geq n-\ell(n,W_T).\]
\end{proposition}

\begin{proof}
    Let $k=r_{odd}(n,\cF_T)$. It is easy to see that $k=1$ if and only if $T=\emptyset$, in which case $W_T=\{n\}$ and we trivially have $\ell(n,W_T)=n-1$. 

    Assume now that $k>1$. Let $G_1,\ldots,G_k$  be a $k$-colouring of $K_n$ such that, for every $t\in T$, every copy of~$K_{t,n-t}$ contains an odd colour class. Let $S_i$ be the set of vertices of odd degree in the graph $G_i$. Note that the Handshaking Lemma implies that $|S_i|$ is even for all $i\in[k]$. For $i\in[k-1]$, let $\mathbf{s}_i\in\mathbb{F}_2^n$ be the characteristic vector of $S_i$, and let $\cC=\mathrm{span}(\mathbf{s}_1,\ldots,\mathbf{s}_{k-1})$. Note that $\cC$ is an even linear code of length $n$ and dimension at most $k-1$. 

    Fix $t\in T$. By definition of the colouring, for every partition $A\sqcup B=[n]$ with $|A|=t$ and $|B|=n-t$, there exists an $i\in[k]$ such that $e_{G_i}(A,B)$ is odd. Since $t(n-t)$ is even, there must in fact be two odd colour classes, and hence we may take $i \in [k-1]$. Now, since $e_{G_i}(A,B)=\sum_{v\in A}\deg_{G_i}(v)-2e_{G_i}(A)$, for $e_{G_i}(A,B)$ to be odd, we must have that $|A \cap S_i|$ is odd. As $|S_i|$ is even, it then follows that $|B \cap S_i|$ is odd as well. By considering the characteristic vectors of $A$ and $B$, it follows that, for every vector $\mathbf{a}$ of weight $t$ or $n-t$, there exists some $i\in[k-1]$ such that $\mathbf{a}\cdot \mathbf{s}_i = 1$. Therefore, $\mathbf{a}\not\in\cC^{\perp}$, and thus $\cC^{\perp}$ contains no codeword of weight $t$ or $n-t$, for every $t \in T$. 

    Given that $\cC$ is an even linear code, we know that $\mathbf{1}_n\in\cC^{\perp}$. Let $\mathbf{v}_1,\dots,\mathbf{v}_r$ be a basis for $\cC^{\perp}$ with~$\mathbf{v}_r = \mathbf{1}_n$. Note that, since $\cC$ has dimension at most $k-1$, $\cC^{\perp}$ has dimension at least $n-k+1$. Let $\cD = \mathrm{span}(\mathbf{v}_1,\ldots,\mathbf{v}_{r-1})$, and observe that~$\cD$ has dimension at least $n-k$. Since $\cC^\perp$ contains no codeword of length~$t$ or~$n-t$, and $\mathbf{1}_n \notin \cD$, it follows that $\cD$ contains no codeword with weight in $W_T$. Hence 
    \[\ell(n,W_T)\geq \dim \cD \ge n-k = n-r_{odd}(n,\cF_T).\qedhere\]

\end{proof}

The following proposition implies the other direction, namely that $r_{odd}(n,\cF_T)\leq n-\ell(n,W_T)$.
\begin{proposition}\label{prop:GeneralEquivalence_Upper}
For integers $n>k\geq 1$ and a set $T\subseteq [n/2]$, assume that there exists a linear binary code $\cA$ of length $n$ and dimension $k$ such that $\cA$ contains no codeword of weight in $W_T$. Then
\[r_{odd}(n,\cF_T)\leq n-k. \]
\end{proposition}

\begin{proof}
Since $\cA$ contains no codeword of weight $n$, we have $\mathbf{1}_n\not\in\cA$. Consider the $(k+1)$-dimensional linear code $\cB$ obtained from $\cA$ by adding the codeword $\mathbf{1}_n$. Since $\cA$ has length~$n$ and contains no words of weight $t$ or $n-t$ for every $t\in T$, it follows that~$\cB$ has the same properties. Finally, let $\cC=\cB^{\perp}$. Given that $\mathbf{1}_n\in\cB$, we know that $\cC$ is an even linear binary code of dimension $n-k-1$.

Let $\mathbf{s}_1,\ldots,\mathbf{s}_{n-k-1}\in \mathbb{F}_2^n$ be a basis for $\cC$. Up to row operations and a permutation of $[n]$, we can select this basis in such a way that 
\begin{equation}\label{eq:MatrixFormCode}
\begin{pmatrix}
    \mathbf{s}_1\\\mathbf{s}_2\\\vdots\\\mathbf{s}_{n-k-1}
\end{pmatrix}
= (I_{n-k-1}|M),    
\end{equation}
for some matrix $M\in \mathbb{F}_2^{(n-k-1)\times (k+1)}$. For each $i\in[n-k-1]$, let $S_i = \supp(\mathbf{s}_i)$.

It follows from~\cref{eq:MatrixFormCode} that $i\in S_i$ for all $i\in [n-k-1]$ and $i\notin S_j$ for any other $j\in [n-k-1]$. Let $G_i$ be the star-graph on vertex set $S_i$  with centre vertex $i$, and observe that the graphs~$G_i$ for $i\in [n-k-1]$ are all edge-disjoint. Since $\cC$ is an even code, each $S_i$ has even size, and hence each $G_i$ is an odd star. Therefore~$S_i$ is the set of all odd-degree vertices in $G_i$. Finally, let $G_{n-k}=K_n\setminus\bigcup_{i=1}^{n-k-1}E(G_i)$.

Fix $t\in T$ and assume that there exists a partition $A\sqcup B=[n]$, with $|A|=t$ and $|B|=n-t$, such that the complete bipartite graph with vertex classes~$A$ and~$B$ intersects each $G_i$ in an even number of edges, that is, $e_{G_i}(A,B)\equiv 0\pmod 2$  for all $i\in[n-k]$.
Now, we have $\sum_{v\in A}\deg_{G_i}(v) = 2 e_{G_i}(A)+e_{G_i}(A,B)$, and thus  $\sum_{v\in A}\deg_{G_i}(v) \equiv 0 \pmod 2$ for all $i\in [n-k]$. This in turn implies that $|A\cap S_i|\equiv 0\pmod 2$ for all $i\in [n-k-1]$. 
Therefore, the characteristic vector $\mathbf{a}\in \F_2^n$ of the set $A$ has weight $t$ and satisfies~$\mathbf{a}\cdot \mathbf{s}_i=0$ for all $i\in[n-k-1]$. Hence, $\mathbf{a}\in \cC^\perp$, a contradiction, since $\cB=\cC^{\perp}$ contains no codewords of weight in~$W_T$.

Therefore $G_1,\ldots,G_{n-k}$ is a colouring of $K_n$ such that, for every $t \in T$, every copy of $K_{t,n-t}$ has an odd colour class, and hence
\[r_{odd}(n,\cF_T)\leq n-k.\qedhere\]
\end{proof}

\subsection{Coding theory results}\label{sec:PrelimCode}

Enomoto, Frankl, Ito, and Nomura~\cite{enomoto1987codes}, Mazorow~\cite[Chapter 2]{mazorow1991extremal}, and Bassalygo, Cohen, and Zemor~\cite{bassalygo2000codes} proved 
a number of results related to codes avoiding specific weights. In the following theorem, we collect several results that we require.

\begin{theorem}\label{thm:CT_results}\hfill
\begin{enumerate}[label=(\alph*)]
    \item Let $n$ be an integer and $S\subseteq[n]$. Write $m$ for the smallest integer in~$S$ and~$r$ for the smallest \emph{even} integer in~$S$. Then\label{item:CT_simple_lower_bound}
    \[\ell(n, S)\geq \max \set{m-1, r-2}.\]
    \item For any odd integers $n$ and $t$ with $1\leq t < \frac{n}{2}$, \label{item:CT_nOdd_tOdd}
    \[\ell(n,{\set{t,n-t,n}})=n-t-2.\]
    \item For any integer $t\geq 1$,  \label{item:CT_Central}\[\ell(4t,{\set{2t,4t}})=2t-1.\]
    \item For any integers $n$ and $t$ such that $t$ is even and $1\leq t < \frac{n}{2}$,  \label{item:CT_tEven}
    \[t-1\leq \ell(n,{\set{t,n-t,n}})\leq n-t-2.\]
    \item  For any integer $d\geq 0$, there exists a constant $C_d>0$ such that, for all $n\geq 2d$, we have\label{item:CT_teven_nLarge}
    \[ n - d\log n - 3d - 1 \le \ell(n,\{1,\hdots, 2d\} \cup \{ n-2d, \hdots, n\}) \le \ell(n,\{2d,n-2d,n\}) \le n - d \log n + C_d. \]
\end{enumerate}
\end{theorem}

\begin{proof}
A proof of~\ref{item:CT_simple_lower_bound} can be found in Mazorow's PhD thesis~\cite[Lemma~2.3]{mazorow1991extremal}. Given its simplicity, we include the proof for completeness. The  intersection of the hyperplanes $x_{m} = 0$, \dots, $x_n = 0$ in $\F_2^n$ is an $(m-1)$-dimensional linear code of length~$n$ in which every codeword has weight in $[m-1]$ and thus not in~$S$. Similarly, the hyperplanes \mbox{$x_{r}=0$,\dots, $x_n = 0$} together with $x_1+\dots+x_{r-1}=0$ intersect in an \emph{even} \mbox{$(r-2)$-dimensional} linear code whose weights belong to $[r-1]\cap 2\mathbb{Z}$ and are therefore again not in $S$.

Mazorow~\cite[Proposition 2.15]{mazorow1991extremal} also proved~\ref{item:CT_nOdd_tOdd}. Part~\ref{item:CT_Central} follows from Enomoto, Frankl, Ito, and Nomura~\cite[Theorem~1.1]{enomoto1987codes}, see also~\cite[Theorem~2.6]{mazorow1991extremal}. The inequalities of~\ref{item:CT_tEven} can be easily deduced from Propositions 2.15 and 2.17 in~\cite{mazorow1991extremal}, depending on the parity of $n$.

We now prove the last part. 
By monotonicity, we have \[\ell(n,\{1,\ldots,2d\}\cup\{n-2d,\ldots,n\})\leq \ell(n,\set{2d,n-2d,n})\leq \ell(n,\set{2d}).\]
Bassalygo, Cohen, and Zemor~\cite{bassalygo2000codes} proved that $\ell(n,\{2d\})=n-d\log n+O(1)$, which yields the upper bound. To prove the lower bound, they implicitly use a standard modification of BCH codes, which we marginally adapt to our setting. 
Let $s$ be the unique integer such that $2^{s-1}\leq n-2d \leq 2^s-1$, noting that $s \le \log n + 1$.  
Let~$\cA$ be a BCH code~\cite{bose1960class} of length $2^s-1$, dimension $2^s-1-ds$, and minimum weight at least $2d+1$. We {\em shorten} $\cA$, removing the last $2^s-1-(n-2d-1)$ coordinates of every codeword and keeping only the codewords that had a $0$ in each deleted coordinate. Note that we obtain a linear code of length $n-2d-1$, minimum weight at least $2d+1$, and dimension at least $n-2d-ds-1$ (for every deleted coordinate, the dimension decreases by at most $1$). Extend this code by adding  $2d+1$ coordinates set at $0$ to every codeword. We obtain a linear code of length $n$ and dimension at least $n-2d-ds-1$, in which  $2d<\wt(\mathbf{v})<n-2d$ for any $\mathbf{v}\in \cC$. Therefore
\[\ell(n,\{1,\ldots,2d\}\cup\{n-2d,\ldots,n\})\geq n-2d-ds-1 \ge n-d\log n - 3d - 1.\qedhere\]
\end{proof}

\subsection{Odd-Ramsey numbers}

We now use the equivalence result from~\cref{thm:BipartiteGeneralEquivalence} and the coding theory statements from~\cref{sec:PrelimCode} to prove~\cref{cor:SpecialCases}.

\begin{proof}[Proof of~\cref{cor:SpecialCases}]  
    For part~\ref{SpecialCases:odd}, note that $n$ must be odd since $t'(n-t')$ is even for all $t'\in T$. Recalling that $W_T = T \cup \{n-t' : t'\in T\} \cup \{n\}$,  part~\ref{item:CT_simple_lower_bound} of \cref{thm:CT_results} implies that
    $$ \ell(n,W_T) \geq  n-t-2, $$
    since $n-t$ is the smallest even integer in $W_T$. On the other hand,
    $$ \ell(n, W_T) \leq \ell(n, \{t,n-t,n\}) = n-t-2.$$
    by the monotonicity of $\ell$ and part~\ref{item:CT_nOdd_tOdd} of~\cref{thm:CT_results}. Thus, \Cref{thm:BipartiteGeneralEquivalence} gives
    $$ r_{odd}(n,\mathcal{F}_T) = n - \ell(n, W_T) = t+2.$$

    Moving on to part~\ref{SpecialCases:asymp}, we begin with the lower bound. Recall that $2d$ is defined to be the largest even integer in $T$. Since $t(n-t)$ is even for all $t\in T$,  either $n$ is odd or $T$ contains only even integers. It follows immediately from~\cref{thm:BipartiteGeneralEquivalence} and~\cref{thm:CT_results}\ref{item:CT_teven_nLarge}  that there exists a constant~$C_d$ such that
    \begin{equation}\label{eq:lowerbound}
        r_{odd}(n,K_{2d,n-2d}) = n-\ell(n,\{2d,n-2d,n\}) \geq d\log n - C_d. 
    \end{equation}   
    Using the monotonicity of the odd-Ramsey number, we obtain \[r_{odd}(n,\cF_T)\geq r_{odd}(n,K_{2d,n-2d}) \geq d \log n - C_d.\]
     
    We now show the upper bound. Set $n'=n-t+2d \leq n$ and let $\cC$ be the linear code  of length $n'$ giving the lower bound on $\ell(n', \{1,\dots,2d\} \cup \{n'-2d, \dots, n'\})$ in \cref{thm:CT_results}\ref{item:CT_teven_nLarge}. Then $\cC$ has length $n-t+2d$, dimension at least $n-d \log n - t - d - 1$, minimum weight at least $2d+1$ and maximum weight at most $n - t - 1$. In particular, $\cC$ contains no codeword with weight in $\{1,\ldots,2d\}\cup \{n-t,\ldots,n\}$. Therefore, any codeword in $\cC$ with weight in $W_T$ has an odd weight. Let~$\cC'$ be the sub-code of $\cC$ obtained by keeping only the codewords with even weights. Then $\cC'$ is a linear code with length $n-t+2d$ and dimension at least $n- d \log n - t - d - 2$, in which every codeword has weight not in $W_T$. Extend this code by adding  $t-2d$ coordinates set to $0$ to every codeword to obtain
    \[\ell(n,W_T)\geq n-d \log n - t - d - 2,\]
    
    and by~\cref{thm:BipartiteGeneralEquivalence}, we have
    \[r_{odd}(n,\cF_T) =n-\ell(n,W_T)\leq d\log n + t + d + 2.\]

    Finally, part~\ref{SpecialCases:evensmall} follows directly from \Cref{thm:BipartiteGeneralEquivalence} and parts~\ref{item:CT_Central} and~\ref{item:CT_tEven} of~\cref{thm:CT_results}.
\end{proof}

We now present a purely combinatorial proof of~\cref{eq:lowerbound}, which does not use the coding theory connection and may be of independent interest.

\begin{proof}[Alternative proof of~\cref{eq:lowerbound}]
Let $G_1, \dots, G_r$ be a colouring of $K_n$ with $r$ colours such that every copy of~$K_{2d,n-2d}$ contains an odd colour class. For a vertex $v$, define the \emph{type} of $v$ as the vector $\tau_v 
 = (\deg_{G_1}(v) \mod 2, \hdots, \deg_{G_r}(v) \mod 2) \in \mathbb{F}_2^r$.

    Suppose we have a set $A$ of $2d$ vertices such that $\sum_{v \in A} \tau_v = \mathbf{0}_r$. Consider the copy of $K_{2d,n-2d}$ with vertex classes $A$ and $B = [n] \setminus A$. For every $i \in [r]$, we have $e_{G_i}(A,B) = \sum_{v \in A} \deg_{G_i}(v) - 2 e_{G_i}(A) \equiv 0 \mod 2$, since $(\sum_{v \in A} \tau_v)_i = 0$. Thus, we reach a contradiction.

    Hence, we must have a multiset $T'$ of $n$ vectors (counting with multiplicity) in $\mathbb{F}_2^r$ such that no $2d$ of them sum to $\mathbf{0}_r$. Let $T$ be the set of distinct vectors appearing in $T'$. For each vector $\tau \in T$, let $m_{\tau}$ be the multiplicity of $\tau$ in $T'$. Observe that, if $\sum_{\tau} (m_{\tau} - 1) \ge 2d$, then we can find $d$ pairs of repeated vectors, which gives a set of $2d$ vectors whose sum is $\mathbf{0}_r$. Thus, $\sum_{\tau} (m_{\tau} - 1) < 2d$, and it follows that $|T| > n-2d$.

    Now consider the Cayley graph $\Gamma = \Gamma(\mathbb{F}_2^r, T)$, where for every $x \in \mathbb{F}_2^r$ and $\tau \in T$, we add an edge $\{x,x+\tau\}$, and colour it $\tau$. Note that this is a properly edge-coloured graph on $N = 2^r$ vertices, with minimum degree at least $n-2d$. Note that a rainbow cycle of length $2d$ in $\Gamma$ corresponds to $2d$ different vectors in~$T$ summing to $\mathbf{0}_r$. Thus, we may assume $\Gamma$ has no rainbow cycle of length $2d$, and then it follows from a result of Janzer~\cite{janzer2023rainbow} that $e(\Gamma) = O_t(N^{1 + 1/d})$. That implies $n-2d \le C_d' N^{1/d} = C_d' 2^{r/d}$, whence we have~$r \ge d \log n - C_d$, for some constants~$C'_d$ and~$C_d$.
\end{proof}

\section{Fixed complete bipartite graphs}
\label{sec:fixed_complete_bipartite_graphs}

Hitherto we have studied the odd-Ramsey numbers of spanning complete bipartite graphs. We now focus on odd-Ramsey numbers of fixed (that is, non-spanning) complete bipartite graphs, proving~\cref{thm:FixedBipGraphs}. The proof is a modification of the K\H{o}v\'ari--S\'os--Tur\'an proof, tailored to our setting; we first require a definition. For a fixed colouring of a graph $G$, and a vertex $v\in V(G)$, we say a set $S\subseteq V(G)$ of $s$ vertices is an \emph{even $s$-neighbourhood of $v$} if there is at most one colour for which $v$ has odd degree to $S$. 

    \begin{lemma}\label{lem:even-nbhds}
    In every $r$-colouring of $K_n$, every vertex $v$ has at least
    \[ \frac{(n-s)^{\ceil{\frac{s}{2}}} (n-r-s)^{\floor{\frac{s}{2}}}}{r^{\floor{\frac{s}{2}}} s!} \]
    even $s$-neighbourhoods.
    \end{lemma}

We defer the proof of~\cref{lem:even-nbhds} to the end of this section, first proving that it implies~\cref{thm:FixedBipGraphs}.

\begin{proof}[Proof of~\cref{thm:FixedBipGraphs}]
    Suppose we have an $r$-colouring of $K_n$ where every copy of $K_{s,t}$ contains an odd colour class. 
    Assuming~\cref{lem:even-nbhds}, we will double-count the set
    \[ \mathcal{E} = \set{(v,S): v \in V(K_n), S \textrm{ an even $s$-neighbourhood of $v$}}. \]

    By the lemma, since every vertex has many even $s$-neighbourhoods, as long as $r=o(n)$ we have 
    \[ \card{\mathcal{E}} \ge \frac{n (n-s)^{\ceil{\frac{s}{2}}} (n-r-s)^{\floor{\frac{s}{2}}}}{r^{\floor{\frac{s}{2}}} s!} = (1 + o(1)) \frac{n^{s+1}}{r^{\floor{\frac{s}{2}}} s!}. \]

    For an upper bound on $\card{\mathcal{E}}$, we consider how often a set $S$ of $s$ vertices can be an even neighbourhood of a vertex. In the case when $s$ is even, observe that if $S$ is an even $s$-neighbourhood of $v$, then every colour-degree of $v$ into $S$ must be even. Hence, if there is a set $T$ of $t$ such vertices, then $S \cup T$ induces an even copy of $K_{s,t}$. Thus, we must have $\card{\mathcal{E}} \le (t-1) \binom{n}{s}$.

    If $s$ is odd (and therefore $t$ even), then for every vertex $v$ for which $S$ is an even $s$-neighbourhood, there is a unique colour $c_v$ in which $v$ has odd colour-degree to $S$. If we have a set $T$ of $t$ such vertices $v$, such that every colour appears an even number of times in $(c_v : v \in T)$, then $S \cup T$ induces an even copy of~$K_{s,t}$. The existence of such a set $T$ is guaranteed if we have at least $r+t$ vertices $v$ for which $S$ is an even $s$-neighbourhood, since then we can find $\frac{t}{2}$ disjoint pairs of vertices $u,v$ with $c_u = c_v$. Thus, in the odd case we have $\card{\mathcal{E}} \le (r+t-1)\binom{n}{s}$.

    Unifying the cases, we have the bound
    \[ \card{\mathcal{E}} \le tr^{2\parens{\frac{s}{2} - \floor{\frac{s}{2}}}} \frac{n^s}{s!}. \]

    Combining our upper and lower bounds on $\card{\mathcal{E}}$, we obtain
    \[ (1 + o(1)) \frac{n^{s+1}}{r^{\floor{\frac{s}{2}}} s!} \le t r^{2\parens{\frac{s}{2} - \floor{\frac{s}{2}}}} \frac{n^s}{s!}. \]

    Solving for $r$ yields $r^{s - \floor{\frac{s}{2}}} \ge (1 + o(1)) \frac{n}{t}$, which simplifies to give the desired lower bound on $r_{odd}(n, K_{s,t})$.
\end{proof}

To complete the proof, we establish the lemma.
\begin{proof}[Proof of \cref{lem:even-nbhds}]
    We build even $s$-neighbourhoods of $v$ by picking $\floor{\frac{s}{2}}$ (ordered) pairs of monochromatic edges incident to $v$, and then picking one final vertex if needed (when $s$ is odd).

    Suppose we have already picked $j \ge 0$ pairs, and let $d_i^{(j)}$ be the degree of $v$ in the colour $i$ after removing the vertices from the first $j$ pairs. Then, the number of choices for the next pair is given by
    \[ \sum_{i=1}^r d_i^{(j)} \parens{d_i^{(j)} - 1}. \]
    By convexity, this is at least
    \begin{align*}
        r \parens*{ \frac{ \sum_{i=1}^r d_i^{(j)}}{r}}\parens*{ \frac{ \sum_{i=1}^r d_i^{(j)}}{r} - 1} &= r \parens*{\frac{n-2j-1}{r}} \parens*{\frac{n-r-2j-1}{r}} \\
        &\ge \frac{(n-s)(n-r-s)}{r}.
    \end{align*}

    Thus, there are at least $\parens*{ \frac{(n-s)(n-r-s)}{r}}^{\floor{\frac{s}{2}}}$ ways to pick the $\floor{\frac{s}{2}}$ pairs. If $s$ is odd, then we can add one of the remaining $n-s$ vertices, which results in a unique odd colour class. Since each even $s$-neighbourhood can be counted this way in at most $s!$ different orders, it follows that we have at least
    \[ \frac{(n-s)^{\ceil{\frac{s}{2}}} (n-r-s)^{\floor{\frac{s}{2}}}}{r^{\floor{\frac{s}{2}}}s!}\]
    distinct even $s$-neighbourhoods, as required.
\end{proof}

\section{Concluding remarks}
\label{sec:concluding_remarks}

In this article, we studied the odd-Ramsey number for spanning complete bipartite graphs, and subfamilies thereof. In particular, for any integer $n\geq 1$ and any set $T\subseteq[n/2]$ such that $s(n-s)$ is even for every~$s \in T$, if $t=\max T$ and $2d$ is the largest even integer in $T$ (or $d=0$ if $T$ contains only odd integers), it follows from~\cref{cor:SpecialCases} that
\[\max\set*{d\log n-C_d,t+2}\leq r_{odd}(n,\cF_T) \leq d\log n+d+t+2,\] 
for some $C_d$. For constant $d$, these bounds are asymptotically tight: if $t=o(\log n)$, we have $r_{odd}(n,\cF_T)=(d+o(1))\log n$, and if $t=\omega(\log n)$, then $r_{odd}(n,\cF_T)=(1+o(1))t$. However, if $t=\Theta(\log n)$ and $T$ contains an even integer, we can only conclude that $r_{odd}(n,\cF_T)=\Theta(\log n)=\Theta(t)$. It would be desirable to understand $r_{odd}(n,\cF_T)$ better in this regime. When $d$ is allowed to grow with $n$, we do not have a good understanding of the behaviour of the corresponding odd-Ramsey number. The case where $T = \set{2d}$ is of particular interest.

\begin{problem}\label{problem:ExactSpanning}
    Determine the behaviour of $r_{odd}(n, K_{2d,n-2d})$.
\end{problem}\medskip

In the non-spanning case, there is considerable room for further investigation.
For integers $s \leq t$ with $st$ even, we proved that the odd-Ramsey number of~$K_{s,t}$ in~$K_n$ is bounded by
    \[ (1 + o(1)) \parens*{\frac{n}{t}}^{1/\ceil{s/2}}\leq r_{odd}(n,K_{s,t}) \leq O \parens*{n^{\frac{2s+2t-4}{st}}}. \]

For large complete bipartite graphs, if we fix an even $s$ and let $t$ go to infinity, then the exponent of our lower bound asymptotically matches that of the upper bound. However, there is still a considerable gap when $s$ is odd, or when $s$ and $t$ are comparable.

\begin{problem}\label{problem:Kst}
    Let $s,t\geq 3$ be fixed integers such that $st$ is even. What is the asymptotic behaviour of $r_{odd}(n,K_{s,t})$?
\end{problem}

\medskip

In~\cite{yip2024k8}, Yip introduced a natural variant of the odd-Ramsey problem. Given an edge-coloured~$K_n$, a copy of a graph $H$ in $K_n$ is called \emph{unique-chromatic} if there is a colour appearing on exactly one edge of the copy; the colouring is called \emph{$H$-unique} if all copies of $H$ in $K_n$ are unique-chromatic. The smallest number of colours needed for an $H$-unique colouring of $K_n$ is denoted by $r_{u}(n,H)$. For a family of graphs $\cH$, we define  $r_{u}(n,\cH)$ and $\cH$-unique colourings analogously. Note that we trivially have 
\[r_{odd}(n,H)\leq r_{u}(n,H)\leq f\parens*{K_n,H, \left\lfloor \tfrac12 e(H) \right\rfloor+1},\] for any integer $n$ and any graph $H$. 

It is natural to ask how tight the first inequality is. For spanning complete bipartite graphs, our proof of the upper bound in \cref{prop:odd_ramsey_unbalanced_family} yields an $n$-colouring of $K_n$ in which every spanning complete bipartite graph is unique-chromatic. In the even case, we can use a slightly more sophisticated construction to again show that the two Ramsey numbers are the same.
\begin{theorem}\label{thm:unique_ramsey_family}
Let $\cF = \{ K_{t,n-t}\colon t \in [n/2]\}$ be the family of all complete bipartite graphs on $n$ vertices. Then
\[r_{u}(n,\cF) = r_{odd}(n,\cF) = \begin{cases}
    n-1 &\qquad\text{ if $n$ is even},\\
    n &\qquad\text{ if $n$ is odd}.
    \end{cases}\]
\end{theorem}
\begin{proof}
    As explained above, it suffices to show that $r_{u}(n,\cF) \le n-1$ when $n$ is even. Consider the following colouring $G_1,\dots, G_{n-1}$ of $K_n$. For $i\in\{2,\ldots,n/2\}$, let $G_i$ be the single edge $\set{i,n}$. For $i\in\{n/2+1,\ldots,n-1\}$, let $G_i$ consist of the two edges $\set{1,i}$ and $\set{1,n-i+1}$. Finally, let $G_1$ contain all other edges.

    Assume that there exists a spanning complete bipartite graph $K$ with vertex classes $A$ and $B$ intersecting each colour class in either zero or at least two edges. Without loss of generality, assume that $n\in A$. For any $i\in\{2,\ldots,n/2\}$, we have $i\in A$, as otherwise $K$ intersects $G_i$ in the unique edge~$\set{i,n}$, forming a unique colour class. Then, for any $i\in\{2,\ldots,n/2\}$, we also have $(n-i+1)\in A$ as otherwise~$K$ intersects~$G_{n-i+1}$ in either the unique edge $\set{1,i}$ or the unique edge $\set{1,n-i+1}$. Then we obtain $\set{2,\ldots,n}\subseteq A$, and hence $B=\{1\}$; therefore, $K$ intersects~$G_1$ in the unique edge $\set{1,n}$, a contradiction.
\end{proof}\medskip

For subfamilies of $\cF$, our approach to studying odd-Ramsey numbers relies on~\cref{thm:BipartiteGeneralEquivalence}, providing a bridge to coding theory. While distinguishing between odd and even numbers translates elegantly to arithmetic over $\mathbb{F}_2$, there is no similar way to do this for unique colourings. The unique-colouring problem might require different methods, and we wonder to what extent it is similar to the odd-Ramsey problem and propose the following question. 
\begin{problem}
    How far apart can $r_{odd}(n, H)$ and $r_u(n, H)$ be for a graph $H$? For a family of graphs $\cH$?
\end{problem}

We close by showing that the odd- and unique-Ramsey numbers can have a large additive gap. Indeed, or the family $\cK$ of all cliques in $K_n$, Alon~\cite{alon2024graph} proved that $r_{odd}(n,\cK)=\floor{n/2}$. However, the colouring provided for the upper bound is not $\cK$-unique; our final result shows that a unique colouring requires almost twice as many colours.
\begin{theorem}
    For the family $\cK$  of all complete graphs,
    \[r_u(n,\cK)=n-1.\] 
\end{theorem}
\begin{proof}
For the upper bound, it is easy to see that the colouring $G_1,\ldots,G_{n-1}$ defined by $G_i=\set*{\set{i,j}\colon j>i}$ is a $\cK$-unique $(n-1)$-colouring of $K_n$; indeed, for any clique $K$, if $a$ is the second largest vertex in $V(K)$, then~$K$ intersects~$G_a$ in a single edge.

For the lower bound, let $G_1, \dots, G_{n-2}$ be a $(n-2)$-colouring of $K_n$. 
Starting with the clique $K = K_n$, we iterate the following process: whenever we have a colour class with a unique edge $xy$, we remove $x$ from the clique. Each step of the iteration removes one vertex and one colour from the clique. Thus, either the process terminates, in which case we have a clique without a unique colour class, or, after $n-3$ iterations, we are left with a copy of $K_3$ that has only one colour, which is again a clique without a unique colour class.
\end{proof}

\paragraph{Acknowledgements.}
The authors would like to thank Gilles Z\'emor for a helpful clarification on~\cite{bassalygo2000codes}.

S.B.: The research leading to these results was supported by EPSRC, grant no.\ EP/V048287/1.
There are no additional data beyond that contained within the main manuscript.

S.D.: Research supported by Taiwan NSTC grants 111-2115-M-002-009-MY2 and 113-2628-M-002-008-MY4.

K.P.: This project
has received funding from the European Union’s Horizon 2020 research and innovation programme
under the Marie Skłodowska-Curie grant agreement No 101034413.\includegraphics[width=5.5mm, height=4mm]{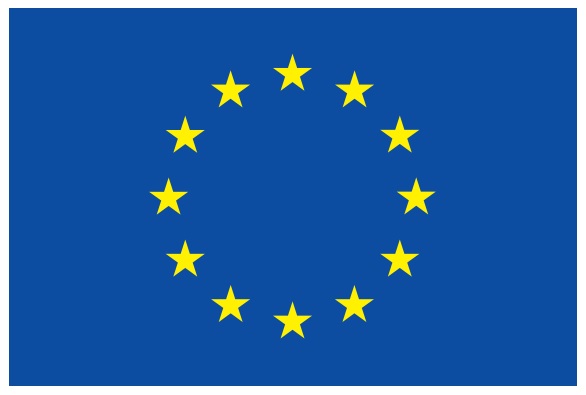} Parts of this research was conducted while K.P. was at the Department of Computer Science,
ETH Z\"urich, Switzerland, supported by Swiss National Science Foundation, grant no. CRSII5 173721.

\bibliographystyle{plain}
\bibliography{bibliography.bib}

\end{document}